\RequirePackage{fix-cm}
\documentclass[envcountsect]{svjour3}                     % onecolumn (standard format)
\smartqed  % flush right qed marks, e.g. at end of proof
\usepackage{graphicx}
\usepackage{amsmath}
\usepackage{amssymb}
\usepackage{mathtools}
\usepackage{caption}
\spnewtheorem{assumption}{Assumption}[section]{\it}{\rm}
\spnewtheorem{algorithm}{Algorithm}{\bf}{\rm}
\spnewtheorem{condition}{Condition}[section]{\it}{\rm}

\journalname{}
\begin{document}

\title{Error Analysis of an Approximate Optimal Policy for a Non-stationary Inventory System with Setup Costs }


%\titlerunning{Error Analysis for an Inventory System}        % if too long for running head
%\titlerunning{} 

\author{Jianyong Liu         \and
        Wei Geng \and
        Xiaobo Zhao.
}


\institute{Jianyong Liu (Corresponding author) \at
              Institute of Applied Mathematics, Academy of Mathematics and System Science,
              Chinese Academy of Sciences, Beijing, 100080, China. \\
              % Tel.: +123-45-678910\\
              % Fax: +123-45-678910\\
              \email{liujy@amss.ac.cn}           %  \\
           \and
           Wei Geng \at
              School of Economics and Management, Southwest Jiaotong University, Chengdu, 610031, China.
           \and
           Xiaobo Zhao \at
              Department of Industrial Engineering, Tsinghua University, Beijing, 100084, China.
}

\date{}

\maketitle

\begin{abstract}
In this paper, we consider a finite horizon non-stationary inventory system with setup costs. We detail an algorithm to find an approximate optimal policy and the basic idea is to use numerical procedure for computing integrals involved in the standard method. We provide analytical error bounds, which converge to zero, between the costs of an approximate optimal policy found in this paper and the optimal policy, which are the main contributions of this paper.  To the best of our knowledge, the error bound results are not found in the literature on finite horizon inventory systems with setup costs. A convergence result for the approximate optimal policy is also provided. In this paper, algorithms to find the above error bounds are also provided. Several numerical examples show that the performance of Algorithms in this paper is satisfactory.

\keywords{Inventory \and Stochastic and continuous demand \and Setup cost \and Approximate optimal policy \and Error bound}
%\subclass{90B05 \and 90C40 \and 49M25}
\subclass{90B05}
\end{abstract}

\section{Introduction}
\label{SecIntro}
In the previous paper(see Geng et al. \cite{GengLiu12}), we discussed an inventory system without setup costs. In this paper, we consider the inventory system with setup costs and obtain analytical error bound results.

The problem of periodically replenishing inventories with stochastic demands is pervasive in retail, wholesale, industry, and service, among others. In many situations, the demand is heavily seasonal or has a significant trend, which requires a non-stationary inventory model. There has been much work on non-stationary inventory models, e.g., Bollapragada and Morton \cite{BollaMorton99}, Morton and Pentico \cite{MortonPenti95}, Sethi and Cheng \cite{SethiCheng97}, Sobel and Zhang \cite{SobelZhang01}, Veinott \cite{Veinott66a,Veinott66b} and Zipkin \cite{Zipkin00}.

In a periodic review inventory system, the demand in each period may be discrete or continuous. If the demand for a particular product is relatively large but the unit of that product is relatively small, then the demand and the stock levels may be described by continuous variables. For example, the sales volume(more than 1000 grams) of gold in a period of time in a store, as a random variable, may be described by a continuous random variable. In addition, it is common that the demand of a product closely follows the normal distribution. Federgruen and Zipkin \cite{FederZipkin85} indicate that inventory control problems are often best described using continuous demands and continuous inventory levels. Work on certain multi-echelon and multi-item systems has heightened interest in continuous-demand models, e.g., Federgruen and Zipkin \cite{FederZipkin84a,FederZipkin84b,FederZipkin84c}, Eppen and Schrage \cite{EppenSchrage81}, etc. Work on inventory systems with continuous demands or general demands includes that by Axs\"{a}ter \cite{Axsater06}, Ehrhardt \cite{Ehrhardt84}, Levi et al. \cite{LeviPal07}, Morton and Pentico \cite{MortonPenti95}, Sahin \cite{Sahin82}, Tsitsiklis \cite{Tsitsik84}, Veinott \cite{Veinott66a,Veinott66b}, and Zipkin \cite{Zipkin00}. In brief, inventory systems with continuous demands are worth studying.

It is well known that, under the appropriate conditions, there exists an optimal (s, S) policy for a finite horizon inventory system(the result also holds for an infinite horizon inventory system with discounted or average cost criterion). However, it is usually difficult to compute the exact optimal (s, S) policy when demands are continuous variables (see Remark \ref{Rmk2.1} in Section \ref{SecModel}). In this paper, we consider a finite horizon non-stationary inventory system with setup costs and independent stochastic demands that may be continuous. We find an approximate optimal (s, S) policy and provide analytical (mathematical) error bounds between the costs of the approximate optimal policy and the optimal policy. The latter constitutes the main contribution in this paper.

For infinite horizon inventory systems with discounted or average cost criterion, the literature involving methods for computing approximate optimal policies includes Ehrhardt \cite{Ehrhardt84}, Federgruen and Zipkin \cite{FederZipkin84c}, Freeland and Porteus \cite{FreelandPorteus80}, Lovejoy \cite{Lovejoy92}, Naddor \cite{Naddor75}, Porteus \cite{Porteus85}, Sahin and Sinha \cite{SahinSinha87} and Schneider and Ringuest \cite{SchneiderRinguest90}.

For finite horizon inventory systems without setup costs, some work has been done on methods for computing approximate optimal policies. Morton and Pentico \cite{MortonPenti95} show that the myopic policy and two near-myopic policies are the upper bounds of the optimal base-stock policy and a near-myopic policy is the lower bound of the optimal policy, and they propose heuristics for computing these bounds as approximate values of the optimal policy. Lovejoy \cite{Lovejoy92} gives an error bound between the costs of the myopic policy and the optimal policy. There is an example (Levi et al. \cite{LeviPal07}) in which the performance of the myopic policy is arbitrarily bad. Lu et al. \cite{LuSong06} consider an inventory system with demand forecasting updates. They provide error bounds between the cost of a given heuristic policy and the optimal cost, but the error bounds do not converge to zero. Levi et al. \cite{LeviPal07} propose an algorithm to compute a so-called dual-balancing policy as an approximate optimal policy using the marginal holding costs and penalty costs. Geng et al. \cite{GengLiu12} provide an analytical error bound between the cost of an approximate optimal policy found in their paper and that of the optimal policy and obtain a convergence result for the approximate optimal policy. However, the method in Geng et al. \cite{GengLiu12} cannot be applied to the inventory system in this paper (see Remark \ref{Rmk2.2} in Section \ref{SecModel} for details).

For the dynamic programming with continuous state space, continuous control space, and a finite set of the disturbance, Bertsekas \cite{Bertsekas75} discretizes the state and control spaces and obtains a sequence of discrete dynamic programming. Then, under certain continuous conditions, it is proved that the solutions of the programmings converge to the solution of the original problem. See also Fox \cite{Fox73} for similar results. It should be pointed out that in Bertsekas \cite{Bertsekas75} and Fox \cite{Fox73}, the cost functions must satisfy the continuity conditions which play the key role in their methods. But, for the inventory system with setup costs, as discussed in this paper, it is obvious that the cost functions do not satisfy the continuity conditions in Bertsekas \cite{Bertsekas75} and Fox \cite{Fox73}. Hence, their methods cannot be applied directly to the inventory system in this paper.

It is more difficult to find an approximate optimal policy for a finite horizon inventory system with setup costs than one without setup costs. There is not much literature to find approximate optimal policies for finite horizon inventory systems with setup costs. Levi et al. \cite{LeviPal07} propose an algorithm to compute a so-called triple-balancing policy as an approximate optimal policy using the marginal holding costs. They provide an error bound between the costs of the triple-balancing policy and the optimal policy, but the error bound does not converge to zero. Bollapragada and Morton \cite{BollaMorton99} propose a myopic heuristic for computing an approximate optimal policy, but they do not provide an analytical error bound between the costs of the approximate optimal policy and the optimal policy. Silver  \cite{Silver78} and Askin  \cite{Askin81} provide a heuristic approach, which  answers  the following questions: i)  is  it  time  to order ? ii)  how  long a period  should  the  replenishment  be expected  to cover ? iii)  how  large should  the  replenishment  be ? But they neither evaluate the performance of the approach nor discuss the error analysis of the approach in their papers.

It is noted that in Bollapragada and Morton \cite{BollaMorton99} and Morton and Pentico \cite{MortonPenti95}, some numerical examples illustrate the errors among the costs of the approximate optimal policies and the optimal cost, but there are no analytical (mathematical) error estimations. There are no convergence results for the approximate optimal policies in Bollapragada and Morton \cite{BollaMorton99}, Lu et al. \cite{LuSong06}, Levi et al. \cite{LeviPal07}, Lovejoy \cite{Lovejoy92} and Morton and Pentico \cite{MortonPenti95}, Silver  \cite{Silver78} and Askin  \cite{Askin81}. Our paper differs from these papers. We provide analytical error bounds, that converge to zero, between the costs of the approximate optimal policy and the optimal policy, which is the main contribution in this paper. To the best of our knowledge, the results on error bounds are not found in the literature on finite horizon inventory systems with setup costs. Then, we obtain the convergence result for the approximate optimal policy from the error analysis results.

In Section \ref{SecModel}, the model description is given. In Section \ref{SecAlgorithm}, we detail an algorithm to find an approximate optimal (s, S) policy. In Section \ref{SecErr}, analysis of the errors between the costs of the approximate optimal policy and the optimal policy is discussed, and the main results are obtained. Section \ref{SecErr} is the main contribution in this paper. The paper concludes with Section \ref{SecConc}.

\section{Model Description}
\label{SecModel}

We consider a finite-horizon non-stationary inventory system with setup costs and independent stochastic demands that may be continuous. Unsatisfied demands are fully backlogged. At the beginning of each period, on observing the inventory level which equals inventory on-hand minus backlogs, a decision about ordering quantity is made. Setup cost is incurred when an order is placed. The ordered goods are received immediately after the decision, i.e., the order delivery lead time is zero. (A model with a stochastic lead time can be transformed into the model with zero lead time. See Ehrhardt \cite{Ehrhardt84}). With the replenished inventory level, the system satisfies the demand during the period. Assume that all parameters can differ from period to period, which means a non-stationary situation. The objective is to minimize the total expected discounted cost.

\textbf{Notation}\\
$t$: Period $t$ in the planning horizon, $t = 0,1,2,\cdots,T$; \\
$\alpha$: Discount factor, $0<\alpha \leq 1$;\\
$c_t$: Unit ordering cost in period $t$, $t = 0,1,2,\cdots,T-1$; \\
$D_t$: Stochastic demand in period $t$
(Assume that the demands $D_0, D_1, D_2, \cdots, D_{T-1}$ are
independent non-negative random variables.);\\
$F_t$: Cumulative distribution function of $D_t$;\\
$K_t$: Setup cost in period $t$ and $K_t \geq 0$.\\

At the beginning of period $t$, the inventory level is $x_t$. With the order quantity decided, the inventory level is shifted to $y_t$. Therefore, $y_t$ equals $x_t$ if the order quantity is zero or is the sum of $x_t$ and the order quantity. With the inventory level $y_t$, the system satisfies demand during period $t$. At the beginning of period $T$, there is no order placed for the observed inventory level $x_T$.

Let $G_t(y)$ denote the expected cost in period $t$, including holding and penalty costs, where $y$ is the inventory level after receipt of order. Note that, in some literature,
\begin{equation*}
    G_t(y) \coloneqq E \left[ h_t \left( \left( y-D_t \right)^+ \right) + p_t \left( \left( D_t-y \right) ^+ \right) \right],
\end{equation*}
where $h_t(\cdot)$ and $p_t(\cdot)$ are the holding cost function and the penalty cost function, respectively, cf. Zipkin \cite{Zipkin00}, Lovejoy \cite{Lovejoy92} and Morton and Pentico \cite{MortonPenti95}.

Let $v_t(x)$ represent the minimal total expected discounted cost from period $t$ to $T$ with an initial inventory level $x$ at the beginning of period $t$. Let $v_T(x) \coloneqq -c_T x$ for any $x$. The factor $c_T$ is called the salvage value, cf. p.373 in Zipkin \cite{Zipkin00}.

%Let $\mathbb{R} \coloneqq ]-\infty,\infty [$. 
Let $\mathbb{R} \coloneqq (-\infty,\infty)$. For $x \in \mathbb{R}$ and $y \in \mathbb{R}$, $v_t(\cdot)$ satisfies the following functional equations for $t = 0,1,2,\cdots,T-1$:
\begin{align}
    v_t(x) & = \min_{y\geq x} \left\{ K_t \delta (y-x) + c_t (y-x) + G_t(y) + \alpha E \left[ v_{t+1} (y-D_t) \right] \right\}, \label{Equ2.1}\\
    v_T(x) & \coloneqq -c_T x,\label{Equ2.2}
\end{align}
where $\delta(z) \coloneqq 1$ if $z>0$ and $\delta(0) \coloneqq 0$.

Let for $x\in \mathbb{R}$ and $t=0,1,2,\cdots,T$,
\begin{equation}\label{Equ2.3}
    V^*_t(x) \coloneqq v_t(x) + c_t x,
\end{equation}
and for $y\in \mathbb{R}$ and $t=0,1,2,\cdots,T-1$,
\begin{equation}\label{Equ2.4}
    C_t(y) \coloneqq (c_t - \alpha c_{t+1})y + G_t(y) + \alpha c_{t+1} E[D_t].
\end{equation}

With the newly defined functions (\ref{Equ2.3}) and (\ref{Equ2.4}), the functional equations (\ref{Equ2.1}) and (\ref{Equ2.2}) are converted into the following equations for $x\in \mathbb{R}$ and $t=0,1,2,\cdots,T-1$:
\begin{align*}
    V^*_t(x) &= \min_{y \geq x} \left\{ K_t \delta(y-x) + C_t(y) + \alpha E \left[ V^*_{t+1} (y-D_t) \right] \right\},\\
    V^*_T(x) &= 0.
\end{align*}

For $y\in \mathbb{R}$ and $t=0,1,2,\cdots,T-1$, define
\begin{equation*}
    H^*_t(y) \coloneqq C_t(y) + \alpha E[V^*_{t+1}(y-D_t)].
\end{equation*}

%We can find the optimal policy $ \{ (s^*_t, S^*_t) \}$(see Theorem 2.1) by  %$H^*_t(x)$ .

\begin{assumption}\label{Asmp2.1}
\begin{description}
  \item[(a).] For $t=0,1,2,\cdots,T-1$, $C_t(y)$ is convex in $y$ with $C_t(y) \rightarrow +\infty$, as $|y| \rightarrow \infty$.
  \item[(b).] For $t=0,1,2,\cdots,T-2$, $K_t \geq \alpha K_{t+1}$.
\end{description}
\end{assumption}

Assumption \ref{Asmp2.1} is common in the literature, cf. Scarf \cite{Scarf60}, Veinott and Wagner \cite{VeinottWagner65}, Veinott \cite{Veinott66a}, Zipkin \cite{Zipkin00} and Sobel and Zhang \cite{SobelZhang01}. Assumption \ref{Asmp2.1} holds throughout this paper. Under Assumption 2.1, there exists the optimal policy  $ \{ (s^*_t, S^*_t) \}$   for the inventory system discussed in this paper(see Theorem 2.1). Then, we detail an algorithm to find the approximate optimal policy $ \{ (s_t, S_t) \}$(see Section 3), and discuss the error analysis and the convergence of the approximate optimal policy(see Section 4).     

For $K\geq 0$, a function $g: \mathbb{R} \rightarrow \mathbb{R}$ is said to be a $K$-convex function, if for any $x, y, z$ satisfying $x\leq y \leq z$,
\begin{equation*}
    g(y) \leq \tau g(x) + (1-\tau)(g(z)+K),
\end{equation*}
where $y=\tau x + (1-\tau) z$, $0\leq \tau \leq 1$.

From Scarf \cite{Scarf60} and Veinott \cite{Veinott66a}, we have the following known theorem.

\begin{theorem}\label{Thm2.1}
    For $t=0,1,2,\cdots,T-1$,
    \begin{description}
     \item[(a)] $H^*_t(x)$ and $V^*_t(x)$ are $K_t$-convex in $x$;
      \item[(b)] $ V^*_t(x)= \begin{cases} H^*_t(S^*_t)+K_t,& x< s^*_t,\\ H^*_t(x),& x\geq s^*_t, \end{cases} $\\
      where
      \begin{align}
          S^*_t & \coloneqq \min\{x|H^*_t(x)=\min_{y\in \mathbb{R}} H^*_t(y) \} \label{Equ2.5} \\
          s^*_t & \coloneqq \max\{x|H^*_t(x)=H^*_t(S^*_t)+K_t, x\leq S^*_t \}; \label{Equ2.6}
      \end{align}
      \item[(c)] the ordering policy $\pi^* \coloneqq \{ (s^*_t, S^*_t)|t=0,1,2,,\cdots,T-1 \} $ is optimal. The ordering policy $\pi^*$ means that for $0\leq t \leq T-1$, if $x_t = x < s^*_t$, order to $S^*_t$ and if $x \geq s^*_t$, do not order.
    \end{description}
\end{theorem}

The standard method to find the optimal policy  $ \{ (s^*_t, S^*_t) \}$  is given in Theorem 2.1.

\begin{remark}\label{Rmk2.1}
    When the demands are continuous, it is usually quite difficult to compute the exact optimal policy $\pi^*=\{ (s^*_t, S^*_t)|t=0,1,2,,\cdots,T-1 \} $ using the standard method described by (\ref{Equ2.5}) and (\ref{Equ2.6}) because the numerical method for computing integrals involved in the standard method may have to be used, which will lead to errors (see Example \ref{Exm3.1}).
\end{remark}

\begin{remark}\label{Rmk2.2}
    Geng et al. \cite{GengLiu12} discuss the case without setup costs. In that case, $V^*_t(x)$ is increasing in $x$. Based on the monotonicity of $V^*_t(x)$, they construct the upper and lower bounds of the optimal cost $v_t(x)$ and obtain the analytical error bound between the costs of the approximate optimal policy and the optimal policy. Then, the convergence result for the approximate optimal policy is given. Note that the monotonicity of $V^*_t(x)$ plays a key role in their method.

    However, in our paper, it is easy to see that $V^*_t(x)$ does not possess the monotonicity when $K_t > 0$. Hence, the method in Geng et al. \cite{GengLiu12} cannot be applied to our inventory system in this paper.
\end{remark}

\section{An Algorithm to Find an Approximate Optimal Policy}
\label{SecAlgorithm}

In this section, we detail an algorithm to find the approximate values of $s^*_t$ and $S^*_t$ (see Theorem \ref{Thm2.1}), namely $s_t$ and $S_t$. The basic idea is to use numerical procedure for computing integrals involved in the standard method described by Theorem \ref{Thm2.1}.

Let $Z_{\theta} \coloneqq \{ z_m | z_m = m\theta, m=0,\pm1,\pm2,\cdots \}$, where $\theta >0$.

\begin{definition}\label{Def3.1}
    For $t=0,1,2,\cdots,T-1$, define
    \begin{align}
        C^m_t & \coloneqq \min \{ y| C_t(y)=\min_{x\in \mathbb{R}} C_t(x) \}; \label{Equ3.1}\\
        S^U_t & \coloneqq \min \{ z_m| C_t(z_m) > C_t(z_{n_0})+K_t,\ \  z_m \geq C^m_t, z_m \in Z_{\theta} \}, \label{Equ3.2}
    \end{align}
    where $z_{n_0} < C^m_t \leq z_{n_0+1}$.

    Let $\bar{S}_0 \coloneqq S^U_0 $.  For $t=1,2,\cdots,T-2$, define
    \begin{equation}\label{Equ3.3}
     \bar{S}_t \coloneqq \max (S^U_t, \bar{S}_{t-1}+\theta).
    \end{equation}
\end{definition}

\begin{definition}\label{Def3.2}
   Take $\bar{Y}_{T-1} $  satisfying
     \begin{align}
        C_{T-1}(\bar{Y}_{T-1})=C_{T-1}(C^m_{T-1})+K_{T-1}, \; \bar{Y}_{T-1}\leq C^m_{T-1} ,\label{Equ3.4}
    \end{align}
define $s_{T-1} \coloneqq  \bar{Y}_{T-1} $.
Define
    \begin{align}
        \bar{I}_{T-1} \coloneqq s_{T-1}.\nonumber
    \end{align}

    For $t=0,1,2,\cdots,T-2$, define
    \begin{align}
        I_t & \coloneqq \max \{ z_m| z_m < \min(\bar{I}_{t+1}-\theta,\;C^m_t), \; z_m \in Z_{\theta} \},\label{Equ3.5}\\
        \bar{I}_t & \coloneqq \max \{ z_m| C_t(z_m)>C_t(I_t)+K_t,\; z_m\leq I_t, z_m\in Z_{\theta} \} + \theta. \label{Equ3.6}
    \end{align}
\end{definition}

With the above definitions, we can define $I_{T-2}$, $\bar{I}_{T-2}$, $I_{T-3}$, $\bar{I}_{T-3}$, $\cdots$, $I_{0}$ and $\bar{I}_{0}$ in order. Obviously, $I_t < C^m_t \leq S^U_t$ and $\bar{I}_t \leq I_t$,  $0\leq t \leq T-2$. 

Via the following  $V_t(y)$  and $ H_t(y)$(see Definition 3.3), we can obtain 
approximate values $ s_t$ and $ S_t $ of  $s^*_t$ and $S^*_t$(see Theorem 2.1).

\begin{definition}\label{Def3.3}
    For $y\in \mathbb{R}$, define $H_{T-1}(y) \coloneqq C_{T-1}(y)$. Define $S_{T-1} \coloneqq C^m_{T-1}$, and
    \begin{equation*}
        V_{T-1}(y) \coloneqq \begin{cases} H_{T-1}(S_{T-1}) + K_{T-1},& y<s_{T-1},\\
         H_{T-1}(y),& y\geq s_{T-1}. \end{cases}
    \end{equation*}
    For $t=0,1,2,\cdots,T-2$ and $y\in \mathbb{R}$, define
    \begin{align}
        H_t(y) & \coloneqq C_t(y)+\alpha \sum_{n=-1}^{\infty} V_{t+1}(y-z_n)f_t(n),\nonumber\\
        V_t(y) & \coloneqq \begin{cases} H_t(S_t) + K_t,& y<s_t,\\ H_t(y),& y\geq s_t, \end{cases} \label{Equ3.7}
    \end{align}
    where $f_t(n) \coloneqq F_t(z_{n+1}) - F_t(z_n)$,
    \begin{equation}\label{Equ3.8}
        S_t \coloneqq \max\left\{ z_m\Bigg| H_t(z_m) = \min_{I_t\leq z_n\leq S^U_t} H_t(z_n),\; I_t\leq z_m\leq S^U_t,\; z_m\in Z_{\theta} \right\},
    \end{equation}
    and
    \begin{equation}\label{Equ3.9}
        s_t \coloneqq \begin{cases} S_t,& K_t=0,\\\min \{ z_m|H_t(z_m)\leq H_t(S_t)+K_t,\; \bar{I}_t\leq z_m\leq S_t,\; z_m\in Z_{\theta} \},& K_t>0. \end{cases}
    \end{equation}
\end{definition}

With the above definitions, we can define $H_{T-2}(y)$, $S_{T-2}$, $s_{T-2}$, $V_{T-2}(y)$, $H_{T-3}(y)$, $S_{T-3}$, $s_{T-3}$, $V_{T-3}(y)$, $\cdots$, $H_{0}(y)$, $S_{0}$, $s_{0}$, and $V_{0}(y)$ in order.

\begin{definition}\label{Def3.4}
     For $K\geq 0$, a function $g:\mathbb{R} \rightarrow \mathbb{R}$ is said to be a sub-$K$-convex function, if for any $x$, $y$, and $z_m \in Z_{\theta}$
satisfying $x\leq y\leq z_m$,
   \begin{equation*}
        g(y)\leq \tau g(x) + (1-\tau)(g(z_m)+K),
    \end{equation*}
    where $y=\tau x + (1-\tau)z_m$, $0\leq \tau \leq 1$.
\end{definition}

\begin{lemma}\label{Lem3.1}
    For $z_m \in Z_{\theta}$, $z_m \leq I_t$, and $t=0,1,2,\cdots,T-2$,
    \begin{equation*}
        H_t(z_m) \geq H_t(I_t).
    \end{equation*}
\end{lemma}
\begin{proof}
    By the definition of $I_t$,
    \begin{equation*}
        I_t + \theta < \bar{I}_{t+1} \leq s_{t+1}.
    \end{equation*}

    Therefore, for $z_m \leq I_t$,
    \begin{equation*}
        H_t(z_m) = C_t(z_m) + \alpha [ H_{t+1}(S_{t+1}) + K_{t+1} ].
    \end{equation*}

    By the definition of $I_t$, $I_t<C^m_t$. Because $C_t(y)$ is convex in $y$, it holds that for $z_m \leq I_t$
    \begin{equation*}
        C_t(z_m) \geq C_t(I_t).
    \end{equation*}

    Therefore, for $z_m \leq I_t$ and $z_m \in Z_{\theta}$,
    \begin{equation*}
        H_t(z_m) \geq C_t(I_t) + \alpha [ H_{t+1}(S_{t+1}) + K_{t+1} ] = H_t(I_t).
    \end{equation*}
    \qed
\end{proof}

The following theorem characterizes the properties of functions $H_t(x)$ and $V_t(x)$.

\begin{theorem}\label{Thm3.1}
    For $t=0,1,2,\cdots,T-1$,
    \begin{description}
      \item[(a)] $H_t(x)$ and $V_t(x)$ are sub-$K_t$-convex functions in $x$;
       \item[(b)] $H_t(S_t) \leq \inf\limits_{z_j \in Z_{\theta}} H_t(z_j)$.
    \end{description}
\end{theorem}
\begin{proof}
    We use the induction method. It is easy to verify that the proposition holds for $t=T-1$.  Assume the proposition holds for $t + 1 (0\leq t \leq T-2)$. From the induction hypothesis and Assumption \ref{Asmp2.1}, it is easy to verify that $H_t(x)$ is a sub-$K_t$-convex function in $x$.

    Because $C_t(x)$ is convex and $S^U_t \geq C^m_t$,
    \begin{equation}\label{Equ3.10}
        C_t(z_j) \geq C_t(S^U_t), \ \ \  z_j \geq S^U_t .
    \end{equation}

    From Formula (\ref{Equ3.7}), the induction hypothesis, Formulas (\ref{Equ3.9}) and (\ref{Equ3.4}), it is easy to verify that
    \begin{equation}\label{Equ3.11}
        V_{t+1}(z_j) - V_{t+1}(z_n) \geq -K_{t+1}, \ \ \  z_j \geq z_n .
    \end{equation}

    Let $z_{n_0} < C^m_t \leq z_{n_0+1}$. From Formula (\ref{Equ3.11}), Assumption \ref{Asmp2.1}, (\ref{Equ3.10}) and (\ref{Equ3.2}), for $z_j \geq S^U_t$,
    \begin{align*}
        H_{t}(z_j) - H_{t}(z_{n_0}) & = C_{t}(z_j) - C_{t}(z_{n_0}) \\
        & \quad + \alpha \sum_{n=-1}^{\infty} \left[ V_{t+1}(z_j-z_n) - V_{t+1}(z_{n_0}-z_n) \right] f_t(n) \\
        & \geq C_{t}(z_j) - C_{t}(z_{n_0}) - \alpha K_{t+1} \\
        & \geq C_{t}(z_j) - C_{t}(z_{n_0}) - K_{t} \\
        & \geq C_{t}(S^U_t) - C_{t}(z_{n_0}) - K_{t} >0.
    \end{align*}
    Therefore, $H_{t}(z_j) > H_{t}(z_{n_0})$,  $z_j \geq S^U_t$.

    Combined with Lemma \ref{Lem3.1}, it holds that $H_t(S_t) = \inf\limits_{z_j \in Z_{\theta}} H_t(z_j)$.

    Similar to the proof of Theorem 9.5.3 in Zipkin \cite{Zipkin00}, we have that $V_t(x)$ is a sub-$K_t$-convex function in $x$. 
Therefore, the proposition also holds for $t$. \qed
\end{proof}

 The algorithm for computing $s_t$ and $S_t$ is detailed below. From the definitions of $H_t(x)$ and $V_t(x)$, for $t=0,1,2,\cdots,T-2$,
\begin{equation}\label{Equ3.12}
    H_t(x) = \begin{cases} C_t(x) + \alpha [ H_{t+1}(S_{t+1}) + K_{t+1} ],& x<s_{t+1}-\theta,\\
     \begin{split} C_t(x) + \alpha [ H_{t+1}(S_{t+1}) + K_{t+1} ]\left( 1-F_t(z_{m+1}) \right) \\ + \alpha \sum_{n=-1}^{m}H_{t+1}(x-z_n)f_t(n),\end{split}& x\geq s_{t+1}-\theta,\end{cases}
\end{equation}
where $z_m \leq x-s_{t+1} \leq z_{m+1}$.

By Formula (\ref{Equ3.12}), we give the following Algorithm 3.1 to find all $s_t$ and $S_t$.

% Algorithm 3.1
\section*{Algorithm 3.1} 
\begin{description}
      \item[Step 0.] Compute all $C^m_t$ and $S^U_t$ ($0 \leq t \leq T-1$) by Formulas (\ref{Equ3.1}) and (\ref{Equ3.2}). Compute $s_{T-1}$ by (\ref{Equ3.4}) and $H_{T-1}(S_{T-1})$. Compute all $I_t$, $\bar{I}_t$ and $\bar{S}_t$ ($0 \leq t \leq T-2$) by (\ref{Equ3.5}), (\ref{Equ3.6}) and (\ref{Equ3.3}). $t \Leftarrow T-2$.
      \item[Step 1.] If $t<0$, stop. Otherwise, compute all $H_t(z_n)$ by (\ref{Equ3.12}), $\bar{I}_t \leq z_n \leq \bar{S}_t$, $z_n \in Z_{\theta}$. Compute $S_t$ and $s_t$ by (\ref{Equ3.8}) and (\ref{Equ3.9}). $t\Leftarrow t-1$, and go to Step 1.
    \end{description}

From $\{ s_t,S_t\}$ found by Algorithm 3.1, we can construct an ordering policy.
We call $\pi = \{ (s_t,S_t)|t=0,1,2,\cdots,T-1 \}$ a policy, which means that for $0 \leq t \leq T-1$, if $x_t=x<s_t$, order to $S_t$, and if $x \geq s_t$ do not order. The cost $v_t^{\pi} (x)$ denotes the total expected discounted cost from period $t$ to $T$ with an initial inventory level $x$ at the beginning of period $t$ using the policy $\pi$.

% The policy $\pi$ can be regarded as an approximation of the optimal policy $\pi^*$.

Example \ref{Exm3.1} below shows that the policy $\pi$ found by Algorithm 3.1  may be a bad policy. On the other hand, in the next section, we shall prove that the policies found by Algorithm 3.1 converge to an optimal policy as $\theta \rightarrow 0$. For that, we call the policy $\pi$  found by Algorithm 3.1 an approximate optimal policy(sometimes the policy $\pi$ is called an approximate policy for short).

\begin{example}\label{Exm3.1}
    Consider a finite horizon inventory system with $T=3$, the salvage value $c_3=1$ and the discounted factor $\alpha = 1$. Other parameters are listed in Table \ref{Table3.1}. In period $t$, the demand distribution function is given as
    \begin{equation*}
        F_t(x) \coloneqq \begin{cases} 0,& x<0,\\
         \bar{K}_t x,& 0\leq x\leq 1/\bar{K}_t, \\
         1,& x> 1/\bar{K}_t, \end{cases}
    \end{equation*}
    where $\bar{K}_0=8$, $\bar{K}_1=3$, and $\bar{K}_2=1$.

    \begin{table}
        \centering
        \caption{Cost parameters in the inventory system}
        \label{Table3.1}
        \begin{tabular}{lccc}
        \hline\noalign{\smallskip}
        Period $t$ & 0 & 1 & 2  \\
        \noalign{\smallskip}\hline\noalign{\smallskip}
        The unit ordering cost $c_t$ & 1 & 1 & 1\\
        The unit holding cost $h_t$ & 1 & 1 & 1\\
        The unit penalty cost $p_t$ & 9 & 4 & 3\\
        Setup cost $K_t$ & 1 & 1 & 1\\
        \noalign{\smallskip}\hline
        \end{tabular}
    \end{table}

    For $\theta = 0.3$, using Algorithm 3.1, we can find an approximate policy
    \begin{equation*}
        \pi = \{ (s_0,S_0),(s_1,S_1),(s_2,S_2) \} =\{(0.3, 0.3), (0.3, 0.6), (0.0429, 0.75)\}.
    \end{equation*}
    Adopting the policy $\pi$, the total expected discounted cost is $v_0^{\pi}(0)=3.994$, and the minimum total expected discounted cost is $v_0(0)= 3.2916$. Thus, the relative error between the costs $v_0^{\pi}(0)$ and $v_0(0)$ is
    \begin{equation*}
        R_{\pi}(0) = \left|\frac{v_0^{\pi}(0)-v_0(0)}{v_0(0)}\right| =21.3\%.
    \end{equation*}

    Consequently, for $\theta = 0.3$, the policy $\pi$ found by Algorithm 3.1  is a bad and unsatisfactory policy.
\end{example}

\section{Error Analysis between $v_t^{\pi}(x)$ and $v_t(x)$}
\label{SecErr}

In this section, we provide the analytical error bounds between the costs of the approximate optimal policy $\pi$ and the optimal policy, which is the main contribution in this paper. We also show the convergence of the approximate optimal policy $\pi$ to the optimal policy as $\theta \rightarrow 0$.

In this section, our method is completely different from the method in Geng et al. \cite{GengLiu12}.  

\begin{definition}\label{Def4.1}
    For $t=0,1,\cdots,T-2$ and $x\in \mathbb{R}$, define $\epsilon_{T-1}(x) \coloneqq 0$, and
    \begin{equation*}
        \epsilon_{t}(x) \coloneqq \begin{cases} \alpha A_{t+1}(S_t) + \alpha E[\epsilon_{t+1}(S_t-D_t)], &x<s_t,\\
             \alpha A_{t+1}(x) + \alpha E[\epsilon_{t+1}(x-D_t)], &x\geq s_t, \end{cases}
    \end{equation*}
    where
    \begin{equation*}
        A_{t}(x) \coloneqq E[V_{t}(x-D_{t-1})] - \sum_{n=-1}^{\infty} V_{t}(x-z_n)f_{t-1}(n), \ \ \ \ \ t=1,\cdots,T-1.
    \end{equation*}
\end{definition}

\begin{lemma}\label{Lem4.1}
    For $t=0,1,\cdots,T-1$ and $x\in \mathbb{R}$, $v_t^{\pi}(x)=V_t(x) - c_t x + \epsilon_t (x)$.
\end{lemma}
\begin{proof}
    This follows by the induction method. \qed
\end{proof}

From Lemma \ref{Lem4.1}, we have the following theorem.

\begin{theorem}\label{Thm4.1}
    For $t=0,1,\cdots,T-1$ and $x\in \mathbb{R}$,
    \begin{equation*}
        0 \leq v_t^{\pi}(x) - v_t(x) = V_t(x) - V^*_t(x) + \epsilon_t (x).
    \end{equation*}
\end{theorem}

From Theorem \ref{Thm4.1}, the error between $v_t^{\pi}(x)$ and the optimal value $v_t(x)$ can be divided into two parts: $V_t(x) - V^*_t(x)$ and $\epsilon_t (x)$. 
Some upper bounds of errors $V_t(x) - V^*_t(x)$ and $\epsilon_t (x)$ are given in subsections \ref{SubSecErr1} and \ref{SubSecErr2}, respectively.

%Some upper bounds of errors $V_t(x) - V^*_t(x)$ and $\epsilon_t (x)$ are given in %subsections \ref{SubSecErr1} and \ref{SubSecErr2}, respectively.

\subsection{The Upper Bounds of Error $V_t(x) - V^*_t(x)$}
\label{SubSecErr1}
In this subsection,
 we shall provide some upper bounds of  $V_t(x) - V^*_t(x)$ . The estimate functions $\psi_t $ and $\bar{\psi}_{t} $ are 
 introduced(see Definitions 4.2 and 4.3 ). Via  the estimate functions  $\psi_t $ and $\bar{\psi}_{t} $  , we introduce the function $\bar{\omega}_t $(see Definition 4.4).  Based on  the properties of functions $\psi_t $ , $\bar{\psi}_{t} $  and $\bar{\omega}_t $ , we prove that $\bar{\omega}_t(x) $ is an upper bound of the error $V_t(x) - V^*_t(x)$(see Theorem 4.2), and provide Algorithm 4.1 to find this upper bound.  Since the structure of  $\bar{\omega}_t(x) $ is rather complicate, we also provide two relatively simple upper bounds of  $V_t(x) - V^*_t(x)$(see Theorem 4.3).

\begin{condition}\label{Cond4.1}
    For any $t\;(0\leq t \leq T-1)$, there exists $\gamma_t \geq 0$ such that for any $x$ and $y$
    \begin{equation*}
        |C_t(x) - C_t(y)| \leq \gamma_t |x-y|.
    \end{equation*}
\end{condition}

We suppose that Condition \ref{Cond4.1} holds below.  Note that Condition \ref{Cond4.1} holds if
\begin{equation*}
    G_t(y) = E \left[ h_t ((y-D_t)^+) + p_t ((D_t-y)^+) \right],
\end{equation*}
where $h_t(\cdot)$ and $p_t(\cdot)$ are the holding cost function and penalty cost function, respectively.

The following functions $\psi_t$, $\varphi_t$, $\bar{\psi}_t$, and $\bar{\varphi}_t$ play key roles in the error analysis.

\begin{definition}\label{Def4.2}
    Define
    \begin{align*}
        \psi_{T-1}(x,y) & \coloneqq \gamma_{T-1}x, \quad \forall x,y,\\
        \varphi_{T-1}(x,y) & \coloneqq \begin{cases} 0,& y<s_{T-1}.\\
             \gamma_{T-1}x,& y\geq s_{T-1}. \end{cases}
    \end{align*}

    For $t=0,1,\cdots,T-2$, define
    \begin{align*}
        \psi_{t}(x,y) & \coloneqq \begin{cases} \gamma_{t}x,& y<s_{t+1}-\theta,\\
             \gamma_{t}x + \alpha \sum_{m=-1}^{n-1} \varphi_{t+1}(x,y-z_m)f_t(m),& y\geq s_{t+1}-\theta, \end{cases}\\
        &\quad \textrm{where} \;z_{n-1} \leq y-s_{t+1} < z_n,\\
        \varphi_{t}(x,y) & \coloneqq \begin{cases} 0,& y<s_t,\\
             \psi_t(y-s_t,y),& y\geq s_t \;\textrm{and}\; y-x < s_t,\\
             \psi_t(x,y),& y\geq s_t \;\textrm{and}\; y-x \geq s_t. \end{cases}
     \end{align*}
\end{definition}

\begin{definition}\label{Def4.3}
     Define
    \begin{align*}
        \bar{\psi}_{T-1}(x,y) & \coloneqq \psi_{T-1}(x,y), \quad \forall x,y,\\
        \bar{\varphi}_{T-1}(x,y) & \coloneqq \varphi_{T-1}(x,y), \quad \forall x,y,
    \end{align*}

    For $t=0,1,\cdots,T-2$, define
    \begin{align*}
        \bar{\psi}_{t}(x,y) & \coloneqq \begin{cases} \gamma_{t}x,& y<s_{t+1}-\theta,\\
             \gamma_{t}x + \alpha \sum_{m=-1}^{n-1} \bar{\varphi}_{t+1}(x,y-z_m)f_t(m),& y\geq s_{t+1}-\theta, \end{cases}\\
        & \quad \textrm{where}\; z_{n-1} \leq y-s_{t+1} < z_n, \\
        \bar{\varphi}_{t}(x,y) & \coloneqq \begin{cases} 0,& y<s_t,\\
             \bar{\psi}_t(y-s_t+\theta,y),& y\geq s_t \;\textrm{and}\; y-x < s_t,\\
             \bar{\psi}_t(x,y),& y\geq s_t \;\textrm{and}\; y-x \geq s_t. \end{cases}
    \end{align*}
\end{definition}

  Via functions ${\psi}_t $ and $\bar{\psi}_t $ , we introduce a function $\bar{\omega}_t $, which is an upper bound of the error $V_t(x) - V^*_t(x) $(see Definition 4.4 and Theorem 4.2).

It is easy to verify that the following lemma holds and the lemma characterizes the properties of functions $\psi_t $, ${\varphi}_{t} $, $\bar{\psi}_t $ and $\bar{\varphi}_{t} $.

\begin{lemma}\label{Lem4.2}
    (a) For $t=0,1,\cdots,T-2$,
    $$\psi_t (x-x',x)=\gamma_t(x-x') + \alpha \sum_{n=-1}^{\infty} \varphi_{t+1}(x-x',x-z_n)f_t(n),\quad \forall x,x', $$
    $$\bar{\psi}_t (x-x',x)=\gamma_t(x-x') + \alpha \sum_{n=-1}^{\infty} \bar{\varphi}_{t+1}(x-x',x-z_n)f_t(n),\quad \forall x,x'.$$

    (b) For $x'\leq x$ and $t=0,1,\cdots,T-1$,
    $$ \varphi_t(x-x',x) \geq 0,\qquad \psi_t(x-x',x)\geq 0,$$
    $$ \bar{\varphi}_t(x-x',x) \geq 0,\qquad \bar{\psi}_t(x-x',x)\geq 0.$$

    (c) (Monotonicity) For $x'\geq x''$, any $x$ and $t=0,1,\cdots,T-1$,
    $$ \varphi_t(x-x',x) \leq \varphi_t(x-x'',x),\qquad \psi_t(x-x',x)\leq \psi_t(x-x'',x),$$
    $$ \bar{\varphi}_t(x-x',x) \leq \bar{\varphi}_t(x-x'',x),\qquad \bar{\psi}_t(x-x',x)\leq \bar{\psi}_t(x-x'',x).$$

    (d) (Monotonicity)
    \begin{description}
      \item[(d1)] If $x\leq z$ and $x-x' \geq 0$, for $t=0,1,\cdots,T-1$,
        \begin{equation*}
            \varphi_t(x-x',x) \leq \varphi_t(x-x',z),\qquad \psi_t(x-x',x)\leq \psi_t(x-x',z),
        \end{equation*}
      \item[(d2)] If $x\leq z$ and $x-z_k \geq 0$, $z_k \in Z_{\theta}$, for $t=0,1,\cdots,T-1$,
        \begin{equation*}
            \bar{\varphi}_t(x-z_k,x) \leq \bar{\varphi}_t(x-z_k,z),\qquad \bar{\psi}_t(x-z_k,x)\leq \bar{\psi}_t(x-z_k,z).
        \end{equation*}
    \end{description}

    (e) For any $x$, any $x'$, and $t=0,1,\cdots,T-1$,
        \begin{equation*}
            \psi_t(x-x',x)\leq \bar{\psi}_t(x-x',x), \qquad \varphi_t(x-x',x) \leq \bar{\varphi}_t(x-x',x).
        \end{equation*}
\end{lemma}

It is easy to see that
\begin{equation}\label{Equ4.1}
    \psi_{T-2} (\theta,z_j) = \begin{cases} \gamma_{T-2}\theta,& n-1<-1,\\
     \gamma_{T-2}\theta +\alpha \sum_{m=-1}^{n-1} \psi_{T-1}(\theta,z_j-z_m)f_{T-2}(m)  ,& n-1\geq -1,\end{cases}
\end{equation}
where $z_{n-1} \leq z_j - s_{T-1} < z_n$, $z_j \in Z_{\theta}$.

From Lemma \ref{Lem4.2} (a),
\begin{equation}\label{Equ4.2}
    \psi_{t} (\theta,z_j) = \begin{cases} \gamma_{t}\theta,& n-2<-1,\\
     \gamma_{t}\theta +\alpha \sum_{m=-1}^{n-2} \psi_{t+1}(\theta,z_j-z_m)f_{t}(m),& n-2\geq -1,\end{cases}
\end{equation}
where $z_{n-2} = z_j - s_{t+1} -\theta$, $0\leq t \leq T-3$, $z_j \in Z_{\theta}$, and
\begin{equation}\label{Equ4.3}
    \bar{\psi}_{t} (\theta,z_j) = \begin{cases} \gamma_{t}\theta,& n-1<-1,\\
     \gamma_{t}\theta +\alpha \sum_{m=-1}^{n-1} \bar{\psi}_{t+1}(\theta,z_j-z_m)f_{t}(m),& n-1\geq -1,\end{cases}
\end{equation}
where $z_{n-1} \leq z_j - s_{t+1} < z_n$, $0\leq t \leq T-2$, $z_j \in Z_{\theta}$.

\begin{lemma}\label{Lem4.3}
    For $x'\leq x$ and $t=0,1,\cdots,T-1$,
    \begin{align*}
        H_t(x) - H_t(x') & \leq \psi_t(x-x',x), \\
        V_t(x) - V_t(x') & \leq \varphi_t(x-x',x).
    \end{align*}
\end{lemma}
\begin{proof}
    This follows by Condition 4.1, Lemma \ref{Lem4.2} (a),  (b) and the induction method. \qed
\end{proof}

\begin{lemma}\label{Lem4.4}
    For $t=0,1,\cdots,T-2$, $H_t(\bar{I}_t - \theta) > H_t(I_t) + K_t$.
\end{lemma}

\begin{proof}
    Let $z_m = \bar{I}_t - \theta$. Obviously, $z_m \leq I_t < \bar{I}_{t+1} - \theta \leq s_{t+1}-\theta$. Thus,
    \begin{align*}
        H_t(z_m) & = C_t(z_m) + \alpha [H_{t+1}(S_{t+1}) + K_{t+1}],\\
        H_t(I_t) & = C_t(I_t) + \alpha [H_{t+1}(S_{t+1}) + K_{t+1}].
    \end{align*}

    From the definition of $\bar{I}_t$,
    \begin{equation*}
        H_t(\bar{I}_t - \theta) - H_t(I_t) = H_t(z_m) - H_t(I_t) = C_t(z_m) - C_t(I_t) > C_t(I_t) + K_t - C_t(I_t) = K_t.
    \end{equation*}
    \qed
\end{proof}

\begin{corollary}\label{Coro4.1}
    For $t=0,1,\cdots,T-1$, $H_t(s_t - \theta) \geq H_t(S_t) + K_t$.
\end{corollary}
\begin{proof}
    This follows from Lemma \ref{Lem4.4} and Theorem \ref{Thm3.1}. \qed
\end{proof}

\begin{lemma}\label{Lem4.5}
    For $x'\leq x$ and $t=0,1,\cdots,T-1$,
    \begin{align*}
        H_t(x') - H_t(x) & \leq \bar{\psi}_t (x-x',x), \\
        V_t(x') - V_t(x) & \leq \bar{\varphi}_t (x-x',x).
    \end{align*}
\end{lemma}
\begin{proof}
    This follows from Condition 4.1, Lemma \ref{Lem4.2} (a),  (b), Corollary \ref{Coro4.1} and the induction method. \qed
\end{proof}

\begin{definition}\label{Def4.4}
    For any $x$, define
    \begin{align*}
        \omega_{T-1} (x) & \coloneqq 0,\\
        \bar{\omega}_{T-1} (x) & \coloneqq 0.
    \end{align*}

    For $t=0,1,\cdots,T-2$, define
    \begin{align}
        \omega_t(x) & \coloneqq \psi_t(\theta,x) - \gamma_t \theta + \alpha \sum_{n=-1}^{\infty}
            \bar{\omega}_{t+1}(x-z_{n})f_t(n), \quad \forall x,\nonumber\\
        \bar{\omega}_t(x) & \coloneqq \begin{cases} \eta_t,& x\leq S^U_t,\nonumber\\
             \max(\eta_t, \omega_t(x)),& x>S^U_t, \end{cases}\nonumber\\
        & \quad \textrm{where} \; \eta_{T-1} \coloneqq 0, \; \eta_t \coloneqq \bar{\psi}_t(\theta,S^U_t)+\omega_t(S^U_t). \label{Equ4.4}
    \end{align}
\end{definition}

We can define $\omega_{T-2}(x)$, $\bar{\omega}_{T-2}(x)$, $\omega_{T-3}(x)$, $\bar{\omega}_{T-3}(x)$, $\cdots$, $\omega_{0}(x)$ and $\bar{\omega}_{0}(x)$ in order. The above $\bar{\omega}_t $ is an upper bound of the error $V_t(x) - V^*_t(x) $(see Theorem 4.2).

\begin{lemma}\label{Lem4.6}
    For $t=0,1,\cdots,T-1$, $\omega_t(x)$ and $\bar{\omega}_t(x)$ are increasing functions in $x$.
\end{lemma}
\begin{proof}
    This follows from Lemma \ref{Lem4.2} (d) and the induction method. \qed
\end{proof}

From Definition \ref{Def4.4},
\begin{equation}\label{Equ4.5}
    \omega_t(x) = \begin{cases} \begin{split}\psi_t(\theta,x) - \gamma_t \theta + \alpha \sum_{m=-1}^{n-1}
     \max ( \eta_{t+1},\; \omega_{t+1}(x-z_m) )f_t(m) \\+ \alpha \eta_{t+1}(1-F_t(z_n)),\end{split}& n-1\geq -1,\\
     \psi_t(\theta,x) - \gamma_t \theta + \alpha \eta_{t+1},& n-1< -1, \end{cases}
\end{equation}
where $z_{n-1} < x-S^U_{t+1} \leq z_n$, $0\leq t \leq T-2$.

\begin{lemma}\label{Lem4.7}
    For $t=0,1,\cdots,T-1$, $S^*_t \leq S^U_t$.
\end{lemma}
\begin{proof}
    The proof is similar to that of Theorem 9.5.6 in Zipkin \cite{Zipkin00}. \qed
\end{proof}

The following theorem gives an upper bound of the error $V_t(x) - V^*_t(x)$.

\begin{theorem}\label{Thm4.2}
    (one of the main results). For $t=0,1,\cdots,T-1$ and any $x$,
    \begin{align*}
        H_t(x) - H^*_t(x) & \leq \omega_t(x),\\
        V_t(x) - V^*_t(x) & \leq \bar{\omega}_t(x).
    \end{align*}
\end{theorem}

The proof of Theorem 4.2 is long. See Appendix.

From Theorem \ref{Thm4.2} and Lemma \ref{Lem4.6}, for $z_{j-1} < x\leq z_j$,
\begin{equation*}
    V_0(x) - V^*_0(x) \leq \bar{\omega}_0(x) \leq \bar{\omega}_0(z_j).
\end{equation*}
That is, $\bar{\omega}_0(z_j)$ is an upper bound of the error $V_0(x) - V^*_0(x)$. For any $z_j \in Z_{\theta}$, we give the following Algorithm 4.1 finds $\bar{\omega}_0(z_j)$.

%\begin{algorithm}\label{Alg4.1}
 \section*{Algorithm 4.1} 
  \begin{description}
      \item[Step 0.]  Compute $L_0=z_j$, $L_m=\max(L_{m-1},S^U_{m-1})+\theta$, $m=1,2,\cdots,T-2$.  Compute $\bar{L}_1 = S^U_0 + \theta$, $\bar{L}_m = \max(\bar{L}_{m-1},S^U_{m-1})+\theta$, $m=2,3,\cdots,T-2$. \  $t \Leftarrow T-2$.
      \item[Step 1.] If $t<1$, go to Step 5. Otherwise, go to Step 2.
      \item[Step 2.] If $s_t+\theta \leq L_t$, compute all $\psi_t(\theta, z_k)$,  $s_t+\theta \leq z_k \leq L_t$, by (\ref{Equ4.1}) and (\ref{Equ4.2}). Compute $\psi_t(\theta,S^U_t)$ and go to Step 3.

          If $s_t+\theta > L_t$, compute $\psi_t(\theta,S^U_t)$ and go to Step 3.
      \item[Step 3.] If $s_t \leq \bar{L}_t$, compute all $\bar{\psi}_t(\theta,z_k)$ by  (\ref{Equ4.3}),  $s_t \leq z_k \leq \bar{L}_t$. Compute $\bar{\psi}_t(\theta,S^U_t)$ and $\eta_{t+1}$ (by (\ref{Equ4.4})). Go to Step 4.

          If $s_t > \bar{L}_t$, compute $\bar{\psi}_t(\theta,S^U_t)$ and $\eta_{t+1}$. Go to Step 4.
      \item[Step 4.] If $S^U_t+\theta \leq L_t$, compute all $\omega_t(z_k)$ by  (\ref{Equ4.5}),  $S^U_t +\theta \leq z_k \leq L_t$. Compute $\omega_t(S^U_t)$. $t \Leftarrow t-1$. Go to Step 1.

          If $S^U_t+\theta > L_t$, compute $\omega_t(S^U_t)$. $t \Leftarrow t-1$. Go to Step 1.
      \item[Step 5.] If $z_j \leq S^U_0$, compute $\psi_0(\theta,S^U_0)$, $\bar{\psi}_0(\theta,S^U_0)$ and $\eta_1$. Compute $\omega_0(S^U_0)$, $\eta_0$ and $\bar{\omega}_0(z_j)$ and then stop.

          If $z_j > S^U_0$, compute $\psi_0(\theta,z_j)$, $\psi_0(\theta,S^U_0)$, $\bar{\psi}_0(\theta,S^U_0)$ and $\eta_1$. Compute $\omega_0(S^U_0)$, $\omega_0(z_j)$, $\eta_0$ and $\bar{\omega}_0(z_j)$ and then stop.
    \end{description}
 %\end{algorithm}

It is rather complicated to find $\bar{\omega}_0(z_j)$ by Algorithm 4.1. Below we give two relatively simple upper bounds of $V_t(x) - V^*_t(x)$.

\begin{lemma}\label{Lem4.8}
    For $t=0,1,2,\cdots,T-2$ and $x \in \mathbb{R}$,
    \begin{equation*}
        \bar{\omega}_t(x) \leq \bar{\psi}_t(\theta,S^U_t) + \bar{\psi}_t(\theta,\max(x,S^U_t)) - \gamma_t \theta + \alpha \bar{\omega}_{t+1} (\theta + \max(x,S^U_t)).
    \end{equation*}
\end{lemma}
\begin{proof}
    By Lemmas \ref{Lem4.6} and \ref{Lem4.2} (e), for $x \in \mathbb{R}$
    \begin{equation*}
        \omega_t(x) = \psi_t(\theta,x) - \gamma_t \theta + \alpha \sum_{n=-1}^{\infty} \bar{\omega}_{t+1}(x-z_n)f_t(n)
        \leq \bar{\psi}_t(\theta,x) - \gamma_t \theta + \alpha \bar{\omega}_{t+1}(x+\theta).
    \end{equation*}

    By Lemma \ref{Lem4.6}, for $x \in \mathbb{R}$
    \begin{equation*}
        \eta_t = \bar{\psi}_t(\theta,S^U_t) + \omega_t(S^U_t) \leq \bar{\psi}_t(\theta,S^U_t) + \omega_t (\max(x,S^U_t)).
    \end{equation*}

    By Lemmas \ref{Lem4.2} (b) and \ref{Lem4.6}, for $x \in \mathbb{R}$
    \begin{equation*}
        \omega_t(x) \leq \bar{\psi}_t(\theta,S^U_t) + \omega_t (\max(x,S^U_t)),
    \end{equation*}
    Thus, for $x \in \mathbb{R}$
    \begin{align*}
        \bar{\omega}_t(x) &\leq \bar{\psi}_t(\theta,S^U_t) + \omega_t (\max(x,S^U_t)) \\
        &\leq \bar{\psi}_t(\theta,S^U_t) + \bar{\psi}_t(\theta,\max(x,S^U_t)) - \gamma_t \theta + \alpha \bar{\omega}_{t+1} (\theta + \max(x,S^U_t)).
    \end{align*}
    \qed
\end{proof}

\begin{lemma}\label{Lem4.9}
    For $z_j \in Z_{\theta}$ and $t=0,1,2,\cdots,T-2$,
    \begin{equation*}
        \bar{\psi}_t (\theta,z_j) \leq \gamma_t \theta + \theta \sum_{n=1}^{T-1-t} \alpha^n \gamma_{t+n} \prod_{m=0}^{n-1} F_{t+m}(z_j + (m+1)\theta - s_{t+m+1}).
    \end{equation*}
\end{lemma}
\begin{proof}
    This follows from (\ref{Equ4.3}) and the induction method. \qed
\end{proof}

 Two relatively simple upper bounds of the error $V_t(x) - V^*_t(x) $ are given in the following Theorem 4.3.

\begin{theorem}\label{Thm4.3}
    For $z_j \in Z_{\theta}$ and $t=0,1,2,\cdots,T-2$,
    \begin{equation*}
        \bar{\omega}_t(z_j) \leq U_t^{\omega}(z_j) \leq \bar{U}_t^{\omega}(\theta),
    \end{equation*}
    where
    \begin{align*}
        U_t^{\omega}(z_j) & \coloneqq \theta \sum_{i=0}^{T-2-t} \alpha^i \gamma_{t+i} + \theta \sum_{i=0}^{T-2-t} \sum_{n=1}^{T-1-t-i} \alpha^{n+i} \gamma_{t+i+n} \\&\quad \cdot \Bigg[ \prod_{m=0}^{n-1} F_{t+i+m}(S^U_{t+i} + (m+1)\theta - s_{t+i+m+1}) \\&\quad + \prod_{m=0}^{n-1} F_{t+i+m}(\max (\mu_t(i),S^U_{t+i}) + (m+1)\theta - s_{t+i+m+1}) \Bigg],
    \end{align*}
    $\mu_t(0) \coloneqq z_j$, $\mu_t(i) \coloneqq \max(\mu_t(i-1),S^U_{t+i-1})+\theta$, \    $i=1,2,\cdots,T-2-t$,
    and
    \begin{equation*}
        \bar{U}_t^{\omega}(\theta) \coloneqq \theta \sum_{i=0}^{T-2-t}\alpha^i \gamma_{t+i} + 2\theta \sum_{i=0}^{T-2-t} \sum_{n=1}^{T-1-t-i} \alpha^{n+i} \gamma_{t+i+n}.
    \end{equation*}
\end{theorem}
\begin{proof}
    This follows from Lemmas \ref{Lem4.8} and \ref{Lem4.9} and the induction method. \qed
\end{proof}

From Theorem \ref{Thm4.2}, Lemma \ref{Lem4.6} and Theorem \ref{Thm4.3}, for $z_{j-1} < x\leq z_j$,
\begin{equation*}
    V_0(x) - V^*_0(x) \leq \bar{\omega}_0(x) \leq \bar{\omega}_0(z_j) \leq U_0^{\omega}(z_j) \leq \bar{U}_0^{\omega}(\theta)
\end{equation*}
That is, $U_0^{\omega}(z_j)$ and $\bar{U}_0^{\omega}(\theta)$ are two rather simple upper bounds of the error $V_0(x) - V^*_0(x)$. Furthermore, $\bar{U}_0^{\omega}(\theta)$ is independent of $x$ and $\bar{U}_0^{\omega}(\theta) \rightarrow 0$ as $\theta \rightarrow 0$. The latter will help us to a convergence result for the approximate optimal policy $\pi$.

\subsection{The Upper Bounds of Error $\epsilon_t (x)$}
\label{SubSecErr2}

In this subsection,
 we shall provide some upper bounds of $\epsilon_t (x)$ . We introduce the function $\bar{\epsilon}_t(x) $(see Definition 4.5).  Based on  the properties of $\bar{\psi}_{t} $  and $\bar{\epsilon}_t $   , we prove that $\bar{\epsilon}_t(x) $  is an upper bound of  $\epsilon_t (x)$(see Theorem 4.4), and provide Algorithm 4.2 to find this upper bound.  Similar to Section 4.1, we also provide two relatively simple upper bounds of $\epsilon_t (x)$(see Theorem 4.5).  
\begin{lemma}\label{Lem4.10}
    If $0 \leq x-x' \leq \theta$, $z_{k-1} < x \leq z_k$, and $t=0,1,2,\cdots,T-1$, then
    \begin{equation*}
        V_t(x') - V_t(x) \leq 2 \bar{\varphi}_t (\theta,z_k).
    \end{equation*}
\end{lemma}
\begin{proof}
    Let $x'\leq z_{k-1}$. From Lemmas \ref{Lem4.5}, \ref{Lem4.2} (c) and \ref{Lem4.2} (d),
    \begin{align*}
        V_t(x') - V_t(x) &= V_t(x') - V_t(z_{k-1}) + V_t(z_{k-1}) - V_t(x)\\
        & \leq \bar{\varphi}_t(z_{k-1}-x',z_{k-1}) + \bar{\varphi}_t(x-z_{k-1},x)\\
        & \leq \bar{\varphi}_t(\theta,z_{k-1}) + \bar{\varphi}_t(x-z_{k-1},z_k)\\
        & \leq \bar{\varphi}_t(\theta,z_{k}) + \bar{\varphi}_t(\theta,z_k) = 2 \bar{\varphi}_t (\theta,z_k).
    \end{align*}

    Let $x' > z_{k-1}$. From Lemmas \ref{Lem4.3}, \ref{Lem4.5}, \ref{Lem4.2} (c), \ref{Lem4.2} (d) and \ref{Lem4.2} (e),
    \begin{align*}
        V_t(x') - V_t(x) &= V_t(x') - V_t(z_{k-1}) + V_t(z_{k-1}) - V_t(x)\\
        & \leq \varphi_t(x'-z_{k-1},x') + \bar{\varphi}_t(x-z_{k-1},x)\\
        & \leq \varphi_t(\theta,x') + \bar{\varphi}_t(x-z_{k-1},z_k)\\
        & \leq \varphi_t(\theta,z_{k}) + \bar{\varphi}_t(\theta,z_k) \\
        & \leq \bar{\varphi}_t(\theta,z_{k}) + \bar{\varphi}_t(\theta,z_k) = 2 \bar{\varphi}_t (\theta,z_k).
    \end{align*}
    \qed
\end{proof}

\begin{lemma}\label{Lem4.11}
    For $z_{k-1} < x \leq z_k$ and $t=0,1,2,\cdots,T-2$,
    \begin{equation*}
        \alpha A_{t+1}(x) \leq 2\bar{\psi}_t(\theta,z_k) - 2\gamma_t \theta.
    \end{equation*}
\end{lemma}
\begin{proof}
    This follows from Lemmas \ref{Lem4.10} and \ref{Lem4.2} (a). \qed
\end{proof}

\begin{definition}\label{Def4.5}
    For any $z_k \in Z_{\theta}$, define
    \begin{equation*}
        \bar{\epsilon}_t (z_k) \coloneqq \begin{cases} \alpha \gamma_{T-1} \theta,& t=T-2,\\
         \begin{split} &\alpha^{T-1-t} \gamma_{T-1} \theta + \sum_{i=0}^{T-3-t} 2\alpha^i \\&\quad \cdot [ \bar{\psi}_{t+i} ( \theta,\max(\lambda_t(i),S_{t+i}) ) - \gamma_{t+i} \theta ],\end{split}& t=0,1,2,\cdots,T-3, \end{cases}
    \end{equation*}
    where $\lambda_t(0) \coloneqq z_k$, $\lambda_t(i) \coloneqq \max(\lambda_t(i-1),S_{t+i-1})+\theta$,  $i=1,2,\cdots,T-3-t$.
\end{definition}

\begin{lemma}\label{Lem4.12}
    For $z_i \leq z_j$ and $t=0,1,2,\cdots,T-2$,
    \begin{equation*}
        \bar{\epsilon}_t(z_i) \leq \bar{\epsilon}_t(z_j).
    \end{equation*}
\end{lemma}
\begin{proof}
    This follows by Lemma \ref{Lem4.2} (d). \qed
\end{proof}

The following theorem gives an upper bound of the error $\epsilon_t(x)$.
\begin{theorem}\label{Thm4.4}
    (one of the main results). For $z_{k-1} <x \leq z_k$ and $t=0,1,2,\cdots,T-2$,
    \begin{equation*}
        \epsilon_t(x) \leq \bar{\epsilon}_t(z_k).
    \end{equation*}
\end{theorem}
 \begin{proof}
    We use the induction method. From Lemmas \ref{Lem4.5} and \ref{Lem4.2} (c), for $y \in \mathbb{R}$,
    \begin{align*}
        \alpha A_{T-1}(y) &= \alpha \sum_{n=-1}^{\infty} \int_{z_n}^{z_{n+1}} [ V_{T-1}(y-\xi) - V_{T-1}(y-z_n) ] dF_{T-2}(\xi) \\
        & \leq \alpha \sum_{n=-1}^{\infty} \int_{z_n}^{z_{n+1}} \bar{\varphi}_{T-1}(\xi -z_n, y-z_n) dF_{T-2}(\xi)\\
        & \leq \alpha \sum_{n=-1}^{\infty} \bar{\varphi}_{T-1}(\theta, y-z_n) f_{T-2}(n)\\
        & \leq \alpha \sum_{n=-1}^{\infty} \gamma_{T-1} \theta f_{T-2}(n) = \alpha \gamma_{T-1} \theta.
    \end{align*}

    Thus, for $z_{k-1} <x \leq z_k$
    \begin{align*}
        \epsilon_{T-2}(x) &= \begin{cases} \alpha A_{T-1}(S_{T-2}),& x<s_{T-2},\\
         \alpha A_{T-1}(x),& x\geq s_{T-2}, \end{cases}\\
         & \leq \alpha \gamma_{T-1} \theta = \bar{\epsilon}_{T-2}(z_k).
    \end{align*}

    That is, the proposition holds for $t=T-2$. Suppose the proposition holds for $t\;(0<t\leq T-2)$. Let $z_{k-1} <x \leq z_k$.

    From Lemma \ref{Lem4.11},
    \begin{align*}
        \epsilon_{t-1}(x) &= \begin{cases} \alpha A_{t}(S_{t-1}) + \alpha E[ \epsilon_t(S_{t-1}-D_{t-1}) ],& x<s_{t-1},\\
         \alpha A_{t}(x) + \alpha E[ \epsilon_t(x-D_{t-1})],& x\geq s_{t-1}, \end{cases}\\
         & \leq \begin{cases} 2\bar{\psi}_{t-1}(\theta,S_{t-1}) - 2 \gamma_{t-1} \theta + \alpha E[ \epsilon_t(S_{t-1}-D_{t-1}) ],& x<s_{t-1},\\
         2 \bar{\psi}_{t-1}(\theta,z_k) - 2 \gamma_{t-1} \theta + \alpha E[ \epsilon_t(x-D_{t-1})],& x\geq s_{t-1}. \end{cases}
    \end{align*}

    From the induction hypothesis and Lemma \ref{Lem4.12}, it is easy to verify that
    \begin{equation*}
        E[ \epsilon_t(S_{t-1}-D_{t-1}) ] \leq \bar{\epsilon}_t (S_{t-1}+\theta) \leq \bar{\epsilon}_t (\max(z_k,S_{t-1})+\theta),
    \end{equation*}
    and
    \begin{equation*}
        E[ \epsilon_t(x-D_{t-1}) ] \leq \bar{\epsilon}_t (z_k+\theta) \leq \bar{\epsilon}_t (\max(z_k,S_{t-1})+\theta),
    \end{equation*}

    Thus, from Lemma \ref{Lem4.2} (d),
    \begin{equation*}
        \epsilon_{t-1}(x) \leq 2\bar{\psi}_{t-1} (\theta, \max(z_k,S_{t-1})) - 2\gamma_{t-1} \theta + \alpha \bar{\epsilon}_t (\max(z_k,S_{t-1})+\theta) = \bar{\epsilon}_{t-1}(z_k).
    \end{equation*}

    That is, the proposition also holds for $t-1$.  \qed
\end{proof}

From Theorem \ref{Thm4.4}, $\epsilon_{0}(x) \leq \bar{\epsilon}_{0}(z_k)$,  $z_{k-1} < x \leq z_k$. That is, $\bar{\epsilon}_{0}(z_k)$ is an upper bound of the error $\epsilon_{0}(x)$. For any $z_k \in Z_{\theta}$, we give the following Algorithm 4.2 finds $\bar{\epsilon}_{0}(z_k)$.

%\section*{Algorithm 4.2}
%\begin{algorithm}\label{Alg4.2}
 \section*{Algorithm 4.2} 
  \begin{description}
   \item[Step 0.] By Definition \ref{Def4.5}, compute $\lambda_0(0)=z_k$, $\lambda_0(1)$, $\lambda_0(2)$, $\cdots$, $\lambda_0(T-3)$. $t \Leftarrow T-3$.
      \item[Step 1.] If $t<0$, go to Step 4. Otherwise, go to Step 2.
      \item[Step 2.] If $s_{t+1} > \max(\lambda_0(t), S_t)+ \theta$, go to Step3.

      If $s_{t+1} \leq \max(\lambda_0(t), S_t)+ \theta$, compute all $\bar{\psi}_{t+1}(\theta, z_j)$ by (\ref{Equ4.3}),   $s_{t+1} \leq z_j \leq \max(\lambda_0(t), S_t)+ \theta$. Go to Step3.
      \item[Step 3.] If $t<T-3$, compute $\bar{\psi}_{t+1}(\theta, \max(\lambda_0(t+1), S_{t+1}))$ by (\ref{Equ4.3}). $t \Leftarrow t-1$ and go to Step 1.

          If $t=T-3$, then $t \Leftarrow t-1$ and go to Step 1.
      \item[Step 4.] Compute $\bar{\psi}_{0}(\theta, \max(\lambda_0(0), S_{0}))$ by  (\ref{Equ4.3}) and $\bar{\epsilon}_0(z_k)$ by Definition \ref{Def4.5}. Stop.
    \end{description}
%\end{algorithm}

It is rather complicated to find $\bar{\epsilon}_0(z_k)$ by Algorithm 4.2. Similar to Subsection \ref{SubSecErr1}, we give two relatively simple upper bounds of $\epsilon_t(x)$.

\begin{theorem}\label{Thm4.5}
    For $z_k \in Z_{\theta}$ and $t=0,1,2,\cdots,T-3$,
    \begin{equation*}
        \bar{\epsilon}_t(z_k) \leq U_t^{\epsilon}(z_k) \leq \bar{U}_t^{\epsilon}(\theta),
    \end{equation*}
    where
    \begin{align*}
        U_t^{\epsilon}(z_k) & \coloneqq \alpha^{T-1-t} \gamma_{T-1} \theta + 2\theta \sum_{i=0}^{T-3-t} \sum_{n=1}^{T-1-t-i} \alpha^{n+i} \gamma_{t+i+n} \\ & \quad \prod_{m=0}^{n-1} F_{t+i+m} (\max(\lambda_t(i),S_{t+i}) + (m+1)\theta - s_{t+i+m+1}),
    \end{align*}
    ( $\lambda_t(i)$ has been defined in Definition \ref{Def4.5}.) and
    \begin{equation*}
        \bar{U}_t^{\epsilon}(\theta) \coloneqq \alpha^{T-1-t} \gamma_{T-1} \theta + 2\theta \sum_{i=0}^{T-3-t} \sum_{n=1}^{T-1-t-i} \alpha^{n+i} \gamma_{t+i+n}.
    \end{equation*}
\end{theorem}
\begin{proof}
    This follows from Lemma \ref{Lem4.9}. \qed
\end{proof}

From Theorems \ref{Thm4.4} and \ref{Thm4.5}, for $z_{k-1} < x \leq z_k$,
\begin{equation*}
    \epsilon_0(x) \leq \bar{\epsilon}_0(z_k) \leq U_0^{\epsilon}(z_k) \leq \bar{U}_0^{\epsilon}(\theta).
\end{equation*}
That is, $U_0^{\epsilon}(z_k)$ and $\bar{U}_0^{\epsilon}(\theta)$ are the two rather simple upper bounds of the error $\epsilon_0(x)$. Furthermore, $\bar{U}_0^{\epsilon}(\theta)$ is independent of $x$ and $\bar{U}_0^{\epsilon}(\theta) \rightarrow 0$ as $\theta \rightarrow 0$. The latter will help us to a convergence result for the approximate optimal policy $\pi$.

\subsection{An Upper Bound of the Relative Error between $v^{\pi}_0(x)$ and $v_0(x)$}
\label{SubSecErr3}

In this subsection,
 the error bounds between $v^{\pi}_0(x)$  and the optimal cost $v_0(x)$  are given by Theorem 4.6,   which is the main contribution in this paper. From the results of Sections 4.1 and 4.2 and  Theorem 4.6, we prove that the approximate policy $\pi$  found by Algorithm 3.1 converges to an optimal policy.  An upper bound of the relative error is given in Formula (\ref{Equ4.9}), and  an approximation of the optimal cost  $v_0(x)$   can be found by Algorithm 4.3.  Furthermore,   we can obtain an upper bound of the relative error. In the end, several numerical examples show that the performance of several algorithms in this paper is satisfactory(see Examples 4.1, 4.2 and 4.3).

\begin{theorem}\label{Thm4.6}
     (one of the main results). For $z_{k-1} < x \leq z_k$,
     \begin{equation}\label{Equ4.7}
        \left| v^{\pi}_0(x) - v_0(x) \right| \leq \bar{\omega}_0(z_k) + \bar{\epsilon}_0(z_k) \leq U^{\omega}_0(z_k) + U^{\epsilon}_0(z_k) \leq \bar{U}^{\omega}_0(\theta) + \bar{U}^{\epsilon}_0(\theta).
     \end{equation}
\end{theorem}
\begin{proof}
    This follows from Theorems \ref{Thm4.1} - \ref{Thm4.5} and Lemma \ref{Lem4.6}. \qed
\end{proof}

Because $\bar{U}^{\omega}_0(\theta) \rightarrow 0$ and $\bar{U}^{\epsilon}_0(\theta) \rightarrow 0$ as $\theta \rightarrow 0$, by Theorem \ref{Thm4.6}, the approximate optimal policy $\pi$ converges to an optimal policy as $\theta \rightarrow 0$.

The relative error between $v^{\pi}_0(x)$ and the optimal cost $v_0(x)$ is defined as
\begin{equation}\label{Equ4.8}
    R_{\pi}(x) \coloneqq \left| \frac{v^{\pi}_0(x) - v_0(x)}{v_0(x)} \right|, \  v_0(x) \neq 0,  x \in \mathbb{R}.
\end{equation}

Roughly, if $R_{\pi}(x)$ is very small, the policy $\pi$ is a good approximation of the optimal policy.

From (\ref{Equ4.7}) and (\ref{Equ4.8}),
\begin{equation}\label{Equ4.9}
    R_{\pi}(x) \leq \frac{\bar{\omega}_0(z_k) + \bar{\epsilon}_0(z_k)}{|v_0(x)|}, \   v_0(x) \neq 0, \  z_{k-1} <x \leq z_k.
\end{equation}

An upper bound of the relative error  $R_{\pi}(x) $ is given by (\ref{Equ4.9}).

From Algorithms 4.1 and 4.2, we can obtain $\bar{\omega}_0(z_k)$ and $\bar{\epsilon}_0(z_k)$. If $v_0(x)$ is known, we can get an upper bound of the relative error $R_{\pi}(x)$. The following theorem gives an approximation of $v_0(x)$.

\begin{theorem}\label{Thm4.7}
    For $z_{k-1} <x \leq z_k$,
    \begin{equation*}
        -\bar{\omega}_0(z_k) \leq v_0(x) - (V_0(x) -c_0 x) \leq \bar{\epsilon}_0(z_k).
    \end{equation*}
\end{theorem}
\begin{proof}
    This follows from Lemma \ref{Lem4.1}, Theorem \ref{Thm4.4}, Theorem \ref{Thm4.2} and Lemma \ref{Lem4.6}. \qed
\end{proof}

The estimate of the error between $v_0(x)$ and $V_0(x) -c_0 x$ is given in Theorem \ref{Thm4.7}. Thus, we can get the approximate value $V_0(x) -c_0 x$ of $v_0(x)$ by computing $V_0(x)$. Furthermore, an upper bound of the relative error $R_{\pi}(x)$ can be found. We give an algorithm  to find $V_0(x)$ below.  Note that $H_t(S_t)$ has been obtained in Algorithm 3.1 for $t=0,1,2,\cdots,T-1$.

% Algorithm 4.3
%\begin{algorithm}\label{Alg4.3}
%     (compute $V_0(x)$).
%\section*{Algorithm 4.3}
%     (compute $V_0(x)$).
 \section*{Algorithm 4.3} 
  \begin{description}
    \item[Step 1.] If $x<s_0$, by (\ref{Equ3.7}), $V_0(x) = H_0(S_0) + K_0$ and stop. If $x \geq s_0$, go to Step 2.
       \item[Step 2.] Compute $\sigma_1$, $\sigma_2$, $\cdots$, $\sigma_{T-1}$, where
           \begin{equation*}
            \sigma_1 \coloneqq \begin{cases} 1,& x \geq s_1 -\theta,\\
             0,& x < s_1 -\theta,\end{cases}
           \end{equation*}
           and for $k=1,2,\cdots,T-2$,
           \begin{equation*}
            \sigma_{k+1} \coloneqq \begin{cases} 1,& \sigma_k=1 \; \textrm{and} \; x \geq s_{k+1} - z_{k+1},\\
             0,& \textrm{otherwise},\end{cases}
           \end{equation*}

           (Note that if $\sigma_k=0$, then $\sigma_m=0$, \   $m\geq k$.)

           $t \Leftarrow T-1$ and go to Step 3.
       \item[Step 3.] If $t=0$, go to Step 4. If $t>0$ and $\sigma_t=0$, $t \Leftarrow t-1$ and go to Step 3.

           If $t>0$ and $\sigma_t=1$, compute all $H_t(x-z_n)$ by (\ref{Equ3.12}), \  $-z_t \leq z_n \leq z_m$, where $z_m \leq x-s_t < z_{m+1}$. $t \Leftarrow t-1$ and go to Step 3.
       \item[Step 4.] Compute $H_0(x)$ by (\ref{Equ3.12}), then $V_0(x) = H_0(x)$. Stop.
     \end{description}
%\end{algorithm}

Examples below exhibit the performance of  Algorithms in this paper. In the examples, the demand distribution functions are generated from the normal, uniform and gamma distributions, respectively.

%Example \ref{Exm4.1} below exhibits how to find an upper bound of the relative %error $R_{\pi}(x)$ using the above algorithms.

\begin{example}\label{Exm4.1}
    Consider a finite horizon inventory system with $T \leq 30 $, $c_{T}=5 $ and the discounted factor $\alpha =1$. The system parameters are as follows: the unit ordering cost $c_t=5$, the unit holding cost $h_t=0.5$, the unit penalty cost $p_t=12$, and the setup cost $K_t=48$,
 for $t=0,1,2,\cdots,T-1$.

    In period $t$, the demand distribution function is
    \begin{equation*}
        F_t(x) \coloneqq \begin{cases} 0,& x<0,\\
         \left(\bar{F}_t(x) - \bar{F}_t(0)\right)/\left(1 - \bar{F}_t(0)\right),& x\geq 0,\end{cases}
     \end{equation*}
    where $\bar{F}_t(x)$ is the normal distribution function with mean $\mu_t$ and standard deviation $\sigma_t$. The mean $\mu_t$ is given in Table 2, and $\sigma_t = \mu_t/5$. 
    
     % \ref{Table4.1}.
     
      \begin{table}
    \centering
    \caption{Mean demand}
   % \label{Table4.5}
    \begin{tabular}{cccccccccccccccc}
    \hline\noalign{\smallskip}
    $t$ & 0 & 1 & 2 & 3 & 4 & 5 & 6 & 7 & 8 & 9 & 10 & 11 & 12 & 13 & 14 \\
    \noalign{\smallskip}\hline\noalign{\smallskip}
    $\mu_t$ & 110 & 40 & 10 & 62 & 12 & 80 & 122 & 130 & 123 & 32 & 13 & 61 & 15 & 87 & 120 \\
    \noalign{\smallskip}\hline\noalign{\smallskip}
    $t$ & 15 & 16 & 17 & 18 & 19 & 20 & 21 & 22 & 23 & 24 & 25 & 26 & 27 & 28 & 29\\
    \noalign{\smallskip}\hline\noalign{\smallskip}
    $\mu_t$ & 115 & 119 & 38 & 14 & 70 & 14 & 86 & 112 & 127 & 123 & 52 & 8 & 73 & 11 & 75 \\
    \noalign{\smallskip}\hline
    \end{tabular}
    \end{table}

    For $T=10 $ and $\theta=0.1$, we can find an approximate optimal policy $\pi$ (see Table 3) by Algorithm 3.1. For an initial inventory level $x_0=0$, we can find $\bar{\omega}_0(0)= 46.501 $, $\bar{\epsilon}_0(0) = 24.463 $ and $V_0(0) =4112.9 $ by Algorithms 4.1, 4.2 and 4.3, respectively.

    \begin{table}
    \centering
    \caption{The approximate optimal policy $\pi = \{ (s_t,S_t)|t=0,1,2,\cdots,9 \}$}
    \label{Table4.2}
    \begin{tabular}{ccccccccccc}
    \hline\noalign{\smallskip}
    $t$ & 0 & 1 & 2 & 3 & 4 & 5 & 6 & 7 & 8 & 9 \\
    \noalign{\smallskip}\hline\noalign{\smallskip}
    $s_t$ & 123.1 & 41.6 & 6.5 & 66.7 & 7.9 & 82.8 & 132 & 141.6 & 138.9 & 28.718 \\
    $S_t$ & 166.3 & 59.7 & 94.9 & 88.3 & 16.2 & 108 & 164.7 & 175.5 & 174.7 & 43.204 \\
    \noalign{\smallskip}\hline
    \end{tabular}
    \end{table}

    By Theorem \ref{Thm4.7},
    \begin{equation}\label{Equ4.10}
        v_0(0) \geq V_0(0) - \bar{\omega}_0(0) =4066.399.
    \end{equation}

    From (\ref{Equ4.9}) and (\ref{Equ4.10}), an upper bound of the relative error $R_{\pi}(0)$ is given as
    \begin{equation*}
        R_{\pi}(0) \leq \frac{\bar{\omega}_0(0) + \bar{\epsilon}_0(0)}{V_0(0)-\bar{\omega}_0(0)} = 1.75\%.
    \end{equation*}

    Similarly, we can obtain the approximate policies and the upper bounds of the relative error for different values of $T$ and $\theta$. Results are given  in Tables 4 and 5.  
   %(see Table \ref{Table4.3}).
 \end{example}

\begin{example}\label{Exm4.2}
    Consider a finite horizon inventory system, whose parameters are the same as the inventory system in Example 4.1.
    
     In period $t$, the demand distribution function is

\begin{equation*}
        F_t(x) \coloneqq \begin{cases} 0,& x<0,\\
         x/(2\mu_t),& 0\leq x\leq 2\mu_t,\\
         1,& x>2\mu_t,\end{cases}
    \end{equation*}
    where $\mu_t$ is given in Table 2. Obviously, $F_t(x)$ is a uniform distribution function, and the mean demand is $\mu_t $.

Similarly to Example 4.1, we can obtain the approximate policies and the upper bounds of the relative error for different values of $T$ and $\theta$. Results are given  in Tables 4 and 5.
\end{example}

\begin{example}\label{Exm4.3}
    Consider a finite horizon inventory system, whose parameters are the same as the inventory system in Example 4.1.
    
     In period $t$, the demand distribution function $F_t(x)$ is a gamma distribution function with parameters $\alpha_t$ and $\beta_t$, where $\alpha_t$=25, $\beta_t$=25/$\mu_t$  and $\mu_t$  is given in Table 2. Obviously, the mean demand is $\mu_t $.  Results are given  in Tables 4 and 5.

\end{example}

\begin{table}
    \centering
    \caption{Results for different values of $T$ and $\theta$ }
    \label{Table4.12}
    \begin{tabular}{lccccc}
    \hline\noalign{\smallskip}
    Example 4.1 \ \ \ \ \ \ \ \ \ \ \ \ \ \ T= & 10 & 15 & 20 & 25 & 30\\
    \noalign{\smallskip}\hline\noalign{\smallskip}
    \ \ \ \ \ \ \ \ \ \ \ \ \ \ \ \ \ \ $\theta$  & 0.1 & 0.09 & 0.09 & 0.1 & 0.09\\
    CPU time(second) for computing  \\the approximate policy & 4.3 & 7.8 & 10.7 & 11.6 & 16.8\\
   % \noalign{\smallskip}\hline\noalign{\smallskip}
    The upper bound of the relative error & 1.75\% & 2\% & 1.87\% & 1.93\% & 1.97\% \\
    CPU time(second) for computing \\the upper bound of the relative error & 4.1 & 8.4 & 11.3 & 12.1 & 17.1 \\
    \noalign{\smallskip}\hline
    Example 4.2 \ \ \ \ \ \ \ \ \ \ \ \ \ \ T= & 10 & 15 & 20 & 25 & 30\\
    \noalign{\smallskip}\hline\noalign{\smallskip}
    \ \ \ \ \ \ \ \ \ \ \ \ \ \ \ \ \ \ $\theta$ & 0.1 & 0.09 & 0.09 & 0.1 & 0.09 \\
    CPU time(second) for computing \\the approximate policy & 1.8 & 3.7 & 5 & 5.3 & 7.9\\
    The upper bound of the relative error & 1.65\% & 1.98\% & 1.9\% & 1.97\% & 2\% \\
    CPU time(second) for computing \\the upper bound of the relative error & 1.6 & 3.9 & 5.8 & 4.9  & 9.5 \\
    \noalign{\smallskip}\hline
    Example 4.3 \ \ \ \ \ \ \ \ \ \ \ \ \ \ T= & 10 & 15 & 20 & 25 & 30\\
    \noalign{\smallskip}\hline\noalign{\smallskip}
    \ \ \ \ \ \ \ \ \ \ \ \ \ \ \ \ \ \ $\theta$ & 0.1 & 0.09 & 0.09 & 0.1 & 0.09 \\
    CPU time(second) for computing  \\the approximate policy & 3.4 & 6.1 & 8.4 & 8.9 & 13.4 \\
    The upper bound of the relative error & 1.77\% & 2\% & 1.91\% & 1.97\% & 2\% \\
    CPU time(second) for computing \\the upper bound of the relative error & 1.8 & 3.5 & 4.9 & 4.1 & 7.8 \\
    \noalign{\smallskip}\hline
    \end{tabular}
    \end{table}

 \begin{table}
    \centering
    \caption{Results for $T=30 $}
    \label{Table4.13}
    \begin{tabular}{lccccc}
    \hline\noalign{\smallskip}
    Example 4.1 & $\theta=1$ & $\theta=0.7$ & $\theta=0.5$ & $\theta=0.3$ & $\theta=0.09$ \\
    \noalign{\smallskip}\hline\noalign{\smallskip}
    CPU time(second) for computing  \\the approximate policy & 1.5 & 2.1 & 2.8 & 4.6 & 16.8 \\
    The upper bound of the relative error & 25.69\% & 16.97\% & 11.72\% & 6.8\% & 1.97\%\\
    CPU time(second) for computing \\the upper bound of the relative error & 1.6 & 2.4 & 3 & 4.9 & 17.1 \\
    \noalign{\smallskip}\hline
    Example 4.2 & $\theta=1$ & $\theta=0.7$ & $\theta=0.5$ & $\theta=0.3$ & $\theta=0.09$ \\
    \noalign{\smallskip}\hline\noalign{\smallskip}
    CPU time(second) for computing  \\the approximate policy & 0.5 & 0.6 & 0.9 & 1.5 & 8.1 \\
    The upper bound of the relative error & 25.8\% & 17.18\% & 11.88\% & 6.91\% & 2\%\\
    CPU time(second) for computing \\the upper bound of the relative error & 0.3 & 0.4 & 0.5 & 0.8 & 9.4 \\
    \noalign{\smallskip}\hline
    Example 4.3 & $\theta=1$ & $\theta=0.7$ & $\theta=0.5$ & $\theta=0.3$ & $\theta=0.09$ \\
    \noalign{\smallskip}\hline\noalign{\smallskip}
    CPU time(second) for computing  \\the approximate policy & 1.1 & 1.7 & 2.1 & 3.5 & 13.4 \\
    The upper bound of the relative error & 26.27\% & 17.39\% & 11.98\% & 6.95\% & 2\%\\
    CPU time(second) for computing \\the upper bound of the relative error & 0.5 & 0.7 & 0.9 & 1.5 & 7.8 \\
    \noalign{\smallskip}\hline
    \end{tabular}
    \end{table}

 Examples 4.1, 4.2 and 4.3 show that the computing time increases as T getting longer(the upper bounds of the relative error do not exceed $2\% $. See Table 4). For T=30, as $\theta$ getting smaller, the computing time is getting longer and the upper bounds of the relative error are getting smaller( tend to 0. See Table 5). We note that the longest computation time is: 16.8 seconds( time for computing the approximate policy) and 17.1 seconds(time for computing the upper bound of the relative error )(see results for T=30 and $\theta$=0.09 in Example 4.1). These examples show that the performance of  Algorithms in this paper is satisfactory.

Our computer conguration is as follows. CPU: Intel(R) Core(TM) i7-4720HQ CPU 2.60GHz; Memory:8GB; and Software: Matlab 2014b.

\section{Conclusions}
\label{SecConc}

In this paper, we consider a finite horizon non-stationary inventory system with backlogging and setup costs. In Section 3, we detail Algorithm 3.1 to find an approximate policy. However, Example \ref{Exm3.1} shows that the approximate policy found by Algorithm 3.1 may be unsatisfactory.

In Section 4, we focus on the error analysis of the approximate policy $\pi$ found by Algorithm 3.1. Using the estimate functions $\psi_t$ and $\bar{\psi}_t$ (see Section \ref{SubSecErr1}), we provide the analytical (mathematical) error bounds, which converge to zero, between the cost $v^{\pi}_0(x)$ of the approximate policy $\pi$ and the optimal cost $v_0(x)$ (Theorems 4.2,4.4 and 4.6), which are the main contributions of this paper. To the best of our knowledge, the error bound results are not found in the literature on finite horizon inventory systems with setup costs.
From the error analysis results, we prove that the approximate policy $\pi$  converges to an optimal policy(Section \ref{SubSecErr3}).

Based on the above results of the error analysis, we can get an upper
bound of the relative error between the cost $v^{\pi}_0(x)$ of the approximate policy $\pi$ found by Algorithm 3.1 and the optimal cost $v_0(x)$. Furthermore, we must be able to obtain a satisfactory approximate optimal policy as $\theta $ getting smaller. Examples 4.1, 4.2 and 4.3 reflect these facts. These examples show that the performance of  Algorithms in this paper is satisfactory.

\section*{Appendix}

\section*{{\rm {\bf The proof of Theorem 4.2}}} 
%    The proof of Theorem 4.2
\begin{proof} 
% \begin{The proof of Theorem 4.2}  

% \begin{The proof of Theorem 4.2}  

    We use the induction method. It is easy to verify that the proposition holds for $t=T-1$. Suppose the proposition holds for $t+1\;(0\leq t\leq T-2)$.
    \begin{align*}
        H_t(x) - H^*_t(x) &= \alpha \sum_{n=-1}^{\infty} V_{t+1}(x-z_n)f_t(n) - \alpha E[V^*_{t+1}(x-D_t)]\\
        &= \alpha \sum_{n=-1}^{\infty} \int_{z_n}^{z_{n+1}} \left[ V_{t+1}(x-z_n)-V_{t+1}(x-\xi) \right] dF_t(\xi)\\
        &\qquad + \alpha \sum_{n=-1}^{\infty} \int_{z_n}^{z_{n+1}} \left[ V_{t+1}(x-\xi) - V^*_{t+1}(x-\xi) \right] dF_t(\xi).
    \end{align*}

    From Lemmas \ref{Lem4.3} and \ref{Lem4.2} (c), it holds that for $z_n \leq \xi \leq z_{n+1}$,
    \begin{equation*}
        V_{t+1}(x-z_n)-V_{t+1}(x-\xi) \leq \varphi_{t+1} (\xi-z_n,x-z_n) \leq \varphi_{t+1}(\theta,x-z_n).
    \end{equation*}

    Thus, by Lemma \ref{Lem4.2} (a),
    \begin{align*}
        & \alpha \sum_{n=-1}^{\infty} \int_{z_n}^{z_{n+1}} \left[ V_{t+1}(x-z_n)-V_{t+1}(x-\xi) \right] dF_t(\xi)\\
        & \quad \leq \alpha \sum_{n=-1}^{\infty} \varphi_{t+1}(\theta,x-z_n)f_t(n) \\
        & \quad = \psi_t(\theta,x)-\gamma_t \theta.
    \end{align*}

    From the induction hypothesis and Lemma \ref{Lem4.6}, it holds that for $\xi \geq z_n$
    \begin{equation*}
        V_{t+1}(x-\xi) - V^*_{t+1}(x-\xi) \leq \bar{\omega}_{t+1}(x-\xi) \leq \bar{\omega}_{t+1}(x-z_n).
    \end{equation*}

    Thus,
    \begin{equation*}
        \alpha \sum_{n=-1}^{\infty} \int_{z_n}^{z_{n+1}} \left[ V_{t+1}(x-\xi) - V^*_{t+1}(x-\xi) \right] dF_t(\xi)
        \leq \alpha \sum_{n=-1}^{\infty} \bar{\omega}_{t+1}(x-z_n) f_t(n),
    \end{equation*}
    and
    \begin{equation}\label{Equ4.6}
        H_t(x) - H^*_t(x) \leq \psi_t(\theta,x)-\gamma_t \theta + \alpha \sum_{n=-1}^{\infty} \bar{\omega}_{t+1}(x-z_n) f_t(n) = \omega_t(x).
    \end{equation}
  
   Below we shall prove that

   \begin{align*}
               V_t(x) - V^*_t(x) & \leq \bar{\omega}_t(x).
    \end{align*}

    Four cases are discussed below.

    (a) Let $x<s_t$ and $x<s^*_t$. Let $z_k < S^*_t \leq z_{k+1}$. From Theorems \ref{Thm2.1} and \ref{Thm3.1}, Lemma \ref{Lem4.5}, Formula (\ref{Equ4.6}), Lemmas \ref{Lem4.6}, \ref{Lem4.7}, \ref{Lem4.2} (d) and \ref{Lem4.2} (c),
    \begin{align*}
        V_t(x) - V^*_t(x) &= H_t(S_t) - H^*_t(S^*_t) \\& \leq H_t(z_k) - H^*_t(S^*_t) \\
        &=H_t(z_k) - H_t(S^*_t) + H_t(S^*_t) - H^*_t(S^*_t)\\
        &\leq \bar{\psi}_t(S^*_t-z_k,S^*_t) + \omega_t(S^*_t)\\
        &\leq \bar{\psi}_t(S^*_t-z_k,S^U_t) + \omega_t(S^U_t)\\
        &\leq \bar{\psi}_t(\theta,S^U_t) + \omega_t(S^U_t)\\
        &= \bar{\omega}_t (x).
    \end{align*}

    (b) Let $x<s_t$ and $x\geq s^*_t$. Two cases are discussed below.

    (b1) Let $x<z_{m_0}=s_t-\theta$. Because $H_t(y)$ is a sub-$K_t$-convex function of $y$ (Theorem \ref{Thm3.1}),
% From Theorem \ref{Thm3.1},
    \begin{equation*}
        H_t(z_{m_0}) \leq \tau H_t(x) + (1-\tau) [H_t(S_t)+K_t], \quad 0< \tau \leq 1.
    \end{equation*}

    By Corollary \ref{Coro4.1},
    \begin{equation*}
        H_t(S_t)+K_t \leq H_t(z_{m_0}) \leq \tau H_t(x) + (1-\tau) [H_t(S_t)+K_t].
    \end{equation*}

    Thus,
    \begin{equation*}
        H_t(S_t)+K_t \leq H_t(x).
    \end{equation*}

    From Theorem \ref{Thm2.1}, (\ref{Equ4.6}) and Lemmas \ref{Lem4.6} and \ref{Lem4.2} (b),
    \begin{equation*}
        V_t(x) - V^*_t(x) = H_t(S_t)+K_t - H^*_t(x) \leq H_t(x)-H^*_t(x) \leq \omega_t(x) \leq \omega_t(S^U_t) \leq \bar{\omega}_t(x).
    \end{equation*}

    (b2) Let $z_{m_0} \leq x < s_t = z_{m_0}+\theta$. From Theorem \ref{Thm2.1}, Corollary \ref{Coro4.1}, Lemma \ref{Lem4.5}, (\ref{Equ4.6}), and Lemmas \ref{Lem4.2} (d), \ref{Lem4.6}, and \ref{Lem4.2} (c),
    \begin{align*}
        V_t(x) - V^*_t(x) &= H_t(S_t) + K_t - H^*_t(x) \\
        & \leq H_t(z_{m_0}) - H^*_t(x) \\
        &=H_t(z_{m_0}) - H_t(x) + H_t(x) - H^*_t(x)\\
        &\leq \bar{\psi}_t(x-z_{m_0},x) + \omega_t(x)\\
        &\leq \bar{\psi}_t(x-z_{m_0},S^U_t) + \omega_t(S^U_t)\\
        &\leq \bar{\psi}_t(\theta,S^U_t) + \omega_t(S^U_t)\\
        &= \bar{\omega}_t (x).
    \end{align*}

    (c) Let $x\geq s_t$ and $x< s^*_t$. Let $z_k < S^*_t \leq z_{k+1}$. From Theorem \ref{Thm3.1}, $H_t(S_t) \leq H_t(z_{k+1})$. Because $H_t(y)$ is a sub-$K_t$-convex function of $y$,
% From Theorem \ref{Thm3.1},

    \begin{align*}
        H_t(x) &\leq \tau H_t(s_t) + (1-\tau) [H_t(z_{k+1})+K_t] \\
        & \leq \tau [H_t(S_t)+K_t] + (1-\tau) [H_t(z_{k+1})+K_t] \\
        & \leq \tau [H_t(z_{k+1})+K_t] + (1-\tau) [H_t(z_{k+1})+K_t]\\
        & = H_t(z_{k+1})+K_t, \qquad 0\leq \tau \leq 1.
    \end{align*}

    From Theorem \ref{Thm2.1}, Lemma \ref{Lem4.3}, (\ref{Equ4.6}), and Lemmas \ref{Lem4.2} (d), \ref{Lem4.7}, \ref{Lem4.6}, \ref{Lem4.2} (c) and \ref{Lem4.2} (e),
    \begin{align*}
        V_t(x) - V^*_t(x) &= H_t(x) - H^*_t(S^*_t) - K_t\\
        & \leq H_t(z_{k+1}) - H^*_t(S^*_t) \\
        & = H_t(z_{k+1}) - H_t(S^*_t) + H_t(S^*_t) - H^*_t(S^*_t)\\
        & \leq \psi_t(z_{k+1}-S^*_t,z_{k+1}) + \omega_t(S^*_t)\\
        & \leq \psi_t(z_{k+1}-S^*_t,S^U_t) + \omega_t(S^U_t)\\
        & \leq \psi_t(\theta,S^U_t) + \omega_t(S^U_t)\\
        & \leq \eta_t \leq \bar{\omega}_t (x).
    \end{align*}

    (d) Let $x \geq s_t$ and $x \geq s^*_t$.

    (d1) Let $x\leq S^U_t$. From Theorem \ref{Thm2.1}, (\ref{Equ4.6}), and Lemmas \ref{Lem4.6} and \ref{Lem4.2} (b),
    \begin{equation*}
        V_t(x) - V^*_t(x) = H_t(x) - H^*_t(x) \leq \omega_t(x) \leq \omega_t(S^U_t) \leq \bar{\psi}_t (\theta,S^U_t) + \omega_t(S^U_t) = \bar{\omega}_t(x).
    \end{equation*}

    (d2) Let $x > S^U_t$. From Theorem \ref{Thm2.1} and (\ref{Equ4.6}),
    \begin{equation*}
        V_t(x) - V^*_t(x) = H_t(x) - H^*_t(x) \leq \omega_t(x) \leq \bar{\omega}_t(x).
    \end{equation*}

    To sum up, the proposition holds for $t$.
    \qed
\end{proof}
% \end{The proof of Theorem 4.2}



\end{document}